\documentclass[3p,10pt,a4paper,twoside,fleqn,sort&compress]{filomat}
\usepackage{amssymb,amsmath,latexsym}
\usepackage[varg]{pxfonts}
\newtheorem{theorem}{Theorem}[section]

\newcommand{\FilVol}{xx}
\newcommand{\FilYear}{yyyy}
\newcommand{\FilPages}{zzz--zzz}
\newcommand{\FilDOI}{(will be added later)}
\newcommand{\FilReceived}{Received: dd Month yyyy; Accepted: dd Month yyyy}
\begin{document}

%\hskip5cm\includegraphics{memo.eps}
%
%\vskip-2.5cm

%%%%%% TO BE ENTERED BY THE AUTHOR(S)
%%%
%%% ENTER TITLE
\title{Hyperbolic k-Fibonacci Quaternions}

%%% AUTHOR(S) FULL NAMES, AND EMAIL ADDRESSES
\author[affil1]{F\"{u}gen Torunbalc{\i}  Ayd{\i}n}
\ead{faydin@yildiz.edu.tr}
%%%
%\author[affil2]{}

%\ead{}
%%% ENTER AUTHOR(S) AFFILIATION(S)
\address[affil1]{%
Yildiz Technical University\\
Faculty of Chemical and Metallurgical Engineering, Department of Mathematical Engineering\\
Davutpasa Campus, 34220, Esenler, Istanbul,  TURKEY}
%\address[affil2]{%

%%% AND CORRESPONDINGLY FOR OTHER AUTHORS, IF THERE ARE MORE AUTHORS
%%% ENTER ABBREVIATED AUTHOR(S) NAMES FOR PAGE HEADINGS
\newcommand{\AuthorNames}{F\"{u}gen Torunbalc{\i}  Ayd{\i}n}
%%% IF THERE ARE MORE THAN TWO AUTHORS WRITE
%%% \newcommand{\AuthorNames}{First Author et al.}
%%%

%%% ENTER MSC, KEYWORDS, RECEIVED, EDITOR, THANKS FOR FINANCIAL SUPPORT FOR RESEARCH
\newcommand{\FilMSC}{Primary 11R52 ; Secondary 11L10, 20G20}
\newcommand{\FilKeywords}{(Fibonacci number, k-Fibonacci number, k-Fibonacci quaternion, dual quaternion, k-Fibonacci dual quaternion.)}
\newcommand{\FilCommunicated}{(Mica S Stankovic)}
%\newcommand{\FilSupport}{Research supported by ... (optionally)}
%%% If you do not want to thank for the financial support of the research, remove
%%% the previous line (i.e., leave \FilSupport undefined)
%%%%%%%%%%%%%%%%%%%%%%%%%%%%%%%%%%%%%%%%%%%%%%%%%%

\begin{abstract}
In this paper, hyperbolic k-Fibonacci quaternions are defined. Also, some algebraic properties of hyperbolic k- Fibonacci quaternions which are connected with hyperbolic numbers and k-Fibonacci numbers are investigated. Furthermore, d'Ocagne's identity, the Honsberger identity, Binet's formula, Cassini's identity and Catalan's identity for these quaternions are given. \\
\end{abstract}

\maketitle

%%%%%% THIS PART MUST BE PLACED IMMEDIATELY AFTER THE \maketitle COMMAND
%%%%%% BACK TO ORIGINAL FOOTNOTES
\makeatletter
\renewcommand\@makefnmark%
{\mbox{\textsuperscript{\normalfont\@thefnmark)}}}
\makeatother
%%%%%%

\section{Introduction}

\par The quaternions constitute an extension of complex numbers into a four-dimensional space and can be considered as four-dimensional vectors, in the same way that complex numbers are considered as two-dimensional vectors.\\
Quaternions were first described by Irish mathematician  Hamilton in 1843. Hamilton \cite{book.1} introduced a set of quaternions which can be represented as  
\begin{equation}\label{E1}
H=\left\{ \,q={q}_{0}+i\,{q}_{1}+j\,{q}_{2}+k\,{q}_{3}\,\left. \right|\, {q}_{0},\,{q}_{1},\,{q}_{2},\,{q}_{3}\in \mathbb R\, \right\}
\end{equation}
where
\begin{equation}\label{E2}
{i}^{2}={j}^{2}={k}^{2}=-1\,,\ \ i\ j=-j\ i=k\,,\quad j\ k=-k \ j=i\,,\quad k\ i=-i\ k=j\,.
\end{equation}
\par Several authors worked on different quaternions and  their generalizations. (\cite{article.1}, \cite{article.10}, \cite{article.11}, \cite{article.12}, \cite{article.14}, \cite{article.17})
Now, let's talk about the work done on Fibonacci quaternion and dual Fibonacci quaternion:\par Horadam \cite{article.9} defined complex Fibonacci and Lucas quaternions as follows
\begin{equation*}
{Q}_{n}={F}_{n}+i\,{F}_{n+1}+j\,\,{F}_{n+2}+k\,{F}_{n+3}
\end{equation*} 
and
\begin{equation*}
{K}_{n}={L}_{n}+i\,{L}_{n+1}+j\,\,{L}_{n+2}+k\,{L}_{n+3}
\end{equation*}
where
\begin{equation*}
{i}^{2}={j}^{2}={k}^{2}=-1\,,\ \ i\ j=-j\ i=k\,,\quad j\ k=-k \ j=i\,,\quad k\ i=-i\ k=j\,.
\end{equation*}  
\par In 2012, Hal{\i}c{\i} \cite{article.6} gave generating functions and Binet's formulas for Fibonacci and Lucas quaternions. In 2013, Hal{\i}c{\i} \cite{article.7} defined complex Fibonacci quaternions as follows:
\begin{equation*}
{{H}_{F}}_{C}=\left\{ {R}_{n}={C}_{n}+{e}_{1}\,{C}_{n+1}+{e}_{2}\,\,{C}_{n+2}+{e}_{3}\,{C}_{n+3}\,\left. \right |\,{C}_{n}={F}_{n}+i\,{F}_{n+1},\,\,  {i}^{2}=-1 \right\}
\end{equation*}
where
\begin{equation*}
\begin{aligned}
&{{e}_{1}}^2={{e}_{2}}^2={{e}_{3}}^2={e}_{1}\, {e}_{2}\,{e}_{3}=-1\,, \\ 
&{e}_{1}\,{e}_{2}=-{e}_{2}\,{e}_{1}={e}_{3}\,, 
{e}_{2}\,{e}_{3}=-{e}_{3}\,{e}_{2}={e}_{1}\,,
{e}_{3}\,{e}_{1}=-{e}_{1}\,{e}_{3}={e}_{2},\,\,\, n\geq1.
\end{aligned} 
\end{equation*} 
\par Ramirez \cite{article.16} defined the the k-Fibonacci and the k-Lucas quaternions as follows:
\begin{equation*}
{D}_{k,n}=\{{F}_{k,n}+i\,{F}_{k,n+1}+j\,\,{F}_{k,n+2}+k\,{F}_{k,n+3}\,\left. \right| {F}_{k,n},\, n-th\, \text{ k-Fibonacci number} \},
\end{equation*}
\begin{equation*}
{P}_{k,n}=\{{L}_{k,n}+i\,{L}_{k,n+1}+j\,\,{L}_{k,n+2}+k\,{L}_{k,n+3}\,\left. \right| {L}_{k,n},\, n-th\, \text{ k-Lucas number}  \}
\end{equation*}
where ${\,i,\,j,\,k\,}$ satisfy the multiplication rules (\ref{E2}). \\
\par In 2015, Polatl{\i} K{\i}z{\i}late\c{s} and Kesim \cite{article.15} defined split k-Fibonacci and split k-Lucas quaternions $({M}_{k,n})$ and $({N}_{k,n})$ respectively as follows:
\begin{equation*}
{M}_{k,n}=\{{F}_{k,n}+i\,{F}_{k,n+1}+j\,\,{F}_{k,n+2}+k\,{F}_{k,n+3}\,\left. \right| {F}_{k,n},\, n-th\, \text{ k-Fibonacci number} \}
\end{equation*}
where ${\,i,\,j,\,k\,}$ are split quaternionic units which satisy the multiplication rules
\begin{equation*}
{i}^{2}=-1,\,{j}^{2}={k}^{2}=i\ j\ k=1\,,\ \ i\,j=-j\, i=k,\,j\,k=-k\,j=-i,\,k\,i=-i\,k=j.
\end{equation*}
\par In 2018, K\"{o}sal \cite{article.16} defined hyperbolic quaternions $({K})$ as follows:
\begin{equation*}
K=\left\{ \,q={a}_{0}+i\,{a}_{1}+j\,{a}_{2}+k\,{a}_{3}\,\left. \right|\, {a}_{0},\,{a}_{1},\,{a}_{2},\,{a}_{3}\in \mathbb R\,\, i,j,k\notin \mathbb R \right\}
\end{equation*}
where ${\,i,\,j,\,k\,}$ are hyperbolic quaternionic units which satisy the multiplication rules
\begin{equation*}
{i}^{2}={j}^{2}={k}^{2}=1\,,\ \ i\,j=k=-j\,i,\,\, j\,k=i=-k\,j,\,\, k\,i=j=-i\,k.
\end{equation*}
In this paper, the hyperbolic k-Fibonacci quaternions and the hyperbolic k-Lucas quaternions will be defined respectively, as follows
\begin{equation}\label{E3}
{\mathbb{H}F_{k,n}}=\{{\mathbb{H}F_{k,n}}={F}_{k,n}+\bold{i}\,{F}_{k,n+1}+\bold{j}\,{F}_{k,n+2}+\bold{k}\,{F}_{k,n+3}\,\left. \right| \, {F}_{k,n},\, nth\, \text{k-Fibonacci number} \}
\end{equation}
and
\begin{equation}\label{E4}
{\mathbb{H}L_{k,n}}=\{{\mathbb{H}L_{k,n}}={L}_{k,n}+\bold{i}\,{L}_{k,n+1}+\bold{j}\,{L}_{k,n+2}+\bold{k}\,{L}_{k,n+3}\,\left. \right| \, {L}_{k,n},\, nth\, \text{k-Lucas number}\}
\end{equation}
where
\begin{equation}\label{E5}
{i}^{2}={j}^{2}={k}^{2}=1\,,\ \ i\,j=k=-j\,i,\,\, j\,k=i=-k\,j,\,\, k\,i=j=-i\,k.
\end{equation}
The aim of this work is to present in a unified manner a variety of algebraic properties of both the hyperbolic k-Fibonacci quaternions as well as the k-Fibonacci quaternions and hyperbolic quaternions. In accordance with these definitions, we given some algebraic properties and Binet's formula for hyperbolic k-Fibonacci quaternions. Moreover, some sums formulas and some identities such as d'Ocagne's, Honsberger, Cassini's and Catalan's identities for these quaternions are given.

\newpage

\section{Hyperbolic k-Fibonacci quaternions}

The k-Fibonacci sequence $\{F_{k,n}\}_{n\in\mathbb{N}}$ \cite{article.16} is defined as
\begin{equation}\label{E6}
\left\{\begin{array}{rl}
{F}_{k,0}=&0,\,\,{F}_{k,1}=1 \\
{F}_{k,n+1}=&k\,{F}_{k,n}+\,{F}_{k,n-1},\,\ n\geq1 \\
or \\
\{{F}_{k,n}\}_{n\in\mathbb{N}}=&\{\,0,\,1,\,k,\,k^2+1,\,k^3+2\,k,\,k^4+3\,k^2+1,...\,\}
\end{array}\right.
\end{equation}
Here, ${k}$  is a positive real number.
In this section, firstly hyperbolic k-Fibonacci quaternions will be defined. Hyperbolic k-Fibonacci quaternions are defined by using the k-Fibonacci numbers and hyperbolic quaternionic units as follows
\begin{equation}\label{E7}
{\mathbb H}F_{k,n}=\{q={F}_{k,n}+\bold{i}\,{F}_{k,n+1}+\bold{j}\,{F}_{k,n+2}+\bold{k}\,{F}_{k,n+3}\,\left. \right| {F}_{k,n},\, n-th\,\, \text{k-Fibonacci number} \},
\end{equation}
where
\begin{equation*}
{i}^{2}={j}^{2}={k}^{2}=1\,,\ \ i\,j=k=-j\,i,\,\, j\,k=i=-k\,j,\,\, k\,i=j=-i\,k.
\end{equation*}
Let ${\mathbb{H}F_{k,n}}$ and ${\mathbb{H}F_{k,m}}$ be two hyperbolic k-Fibonacci quaternions such that
\begin{equation}\label{E8}
{\mathbb{H}F_{k,n}}={F}_{k,n}+\bold{i}\,{F}_{k,n+1}+\bold{j}\,\,{F}_{n+2}+\bold{k}\,{F}_{k,n+3}
\end{equation}
and
\begin{equation}\label{E9}
{\mathbb{H}F_{k,m}}={F}_{k,m}+\bold{i}\,{F}_{k,m+1}+\bold{j}\,\,{F}_{k,m+2}+\bold{k}\,{F}_{k,m+3}
\end{equation}
Then, the addition and subtraction of two hyperbolic k-Fibonacci quaternions are defined in the obvious way,
\begin{equation}\label{E10}
\begin{array}{rl}
\mathbb{H}F_{k,n}\pm \mathbb{H}F_{k,m}=&({F}_{k,n}+\bold{i}\,{F}_{k,n+1}+\bold{j}\,{F}_{k,n+2}+\bold{k}\,{F}_{k,n+3})\pm({F}_{k,m}+\bold{i}\,{F}_{k,m+1}+\bold{j}\,{F}_{k,m+2}+\bold{k}\,{F}_{k,m+3}) \\
=&({F}_{k,n}\pm{F}_{k,m})+\bold{i}\,({F}_{k,n+1}\pm{F}_{k,m+1})+\bold{j}\,({F}_{k,n+2}\pm{F}_{k,m+2})+ \bold{k}\,({F}_{k,n+3}\pm{F}_{k,m+3}).
\end{array}
\end{equation}
Multiplication of two hyperbolic k-Fibonacci quaternions is defined by 
\begin{equation}\label{E11}
\begin{array}{rl}
\mathbb{H}F_{k,n}\,\mathbb{H}F_{k,m}=&({F}_{k,n}+i\,{F}_{k,n+1}+j\,{F}_{k,n+2}+k\,{F}_{k,n+3})\,({F}_{k,m}+i\,{F}_{k,m+1}+j\,{F}_{k,m+2}+k\,{F}_{k,m+3}) \\
=&({F}_{k,n}\,{F}_{k,m}+{F}_{k,n+1}\,{F}_{k,m+1}+{F}_{k,n+2}\,{F}_{k,m+2}+{F}_{k,n+3}\,{F}_{k,m+3}) \\
&+\,i\,({F}_{k,n}{F}_{k,m+1}+{F}_{k,n+1}{F}_{k,m}+{F}_{k,n+2}{F}_{k,m+3}-{F}_{k,n+3}{F}_{k,m+2}) \\
&+j\,({F}_{k,n}{F}_{k,m+2}+{F}_{k,n+2}{F}_{k,m}-{F}_{k,n+1}{F}_{k,m+3}+{F}_{k,n+3}{F}_{k,m+1}) \\
&+k\,({F}_{k,n}{F}_{k,m+3}+{F}_{k,n+3}{F}_{k,m}+{F}_{k,n+1}{F}_{k,m+2}-{F}_{k,n+2}{F}_{k,m+1})\neq \mathbb{H}F_{k,m}\,\mathbb{H}F_{k,n}.   
\end{array}
\end{equation}
The scaler and the vector parts of hyperbolic k-Fibonacci quaternion \, ${\mathbb{H}F_{k,n}}$ are denoted by 
\begin{equation}\label{E12}
{S}_{\mathbb{H}F_{k,n}}={F}_{k,n}\,\,$and$ \, \, \, {V}_{\mathbb{H}F_{k,n}}=i\,{F}_{k,n+1}+j\,{F}_{k,n+2}+k\,{F}_{k,n+3}.
\end{equation}
Thus, hyperbolic k-Fibonacci quaternion ${\mathbb{H}F_{k,n}}$  is given by $\mathbb{H}F_{k,n}={S}_{\mathbb{H}F_{k,n}}+{V}_{\mathbb{D}F_{k,n}}$. 
The conjugate of hyperbolic k-Fibonacci quaternion ${\mathbb{H}F_{k,n}}$ is denoted by $\overline{\mathbb{H}F}_{k,n}$ \, and it is
\begin{equation}\label{E13}
\overline{\mathbb{H}F}_{k,n}={F}_{k,n}-i\,{F}_{k,n+1}-j\,{F}_{k,n+2}-k\,{F}_{k,n+3}.
\end{equation}
The norm of hyperbolic k-Fibonacci quaternion ${\mathbb{H}F_{k,n}}$ is defined as follows
\begin{equation}\label{E14}
\| {\mathbb{H}F_{k,n}}\|^2={\mathbb{H}F_{k,n}}\,\,\overline{\mathbb{H}F}_{k,n}=F_{k,n}^2-F_{k,n+1}^2-F_{k,n+2}^2-F_{k,n+3}^2.
\end{equation}
In the following theorem, some properties related to hyperbolic k-Fibonacci quaternions are given.
\begin{theorem} Let ${F}_{k,n}$ and ${\mathbb{H}F_{k,n}}$ be the  n-th terms of  k-Fibonacci sequence $({F}_{k,n})$ and hyperbolic k-Fibonacci quaternion $({\mathbb{H}F_{k,n}})$, respectively. In this case, for \, $n\ge1$ we can give the following relations:
\begin{equation}\label{E15}
{\mathbb{H}F_{k,n+2}}=k\,{\mathbb{H}F_{k,n+1}}+{\mathbb{H}F_{k,n}}
\end{equation}
\begin{equation}\label{E16}
\mathbb{H}F_{k,n}^2=2\,{F}_{k,n}\,\mathbb{H}F_{k,n}-{\mathbb{H}F_{k,n}}\,\,\overline{\mathbb{H}F}_{k,n}
\end{equation}
\begin{equation}\label{E17}
\mathbb{H}F_{k,n}-i\,\mathbb{H}F_{k,n+1}-j\,\mathbb{H}F_{k,n+2}-k\,\mathbb{H}F_{k,n+3}={F}_{k,n}-{F}_{k,n+2}-{F}_{k,n+4}-{F}_{k,n+6}
\end{equation}
\end{theorem}	
\begin{proof}
(\ref{E15}): By the equation (\ref{E7}) we get,
\begin{equation*}
\begin{array}{rl}
\mathbb{H}F_{k,n}+k\,\mathbb{H}F_{k,n+1}=&({F}_{k,n}+i\,{F}_{k,n+1}+j\,{F}_{k,n+2}+k\,{F}_{k,n+3})+k\,({F}_{k,n+1}+i\,{F}_{k,n+2}+j\,{F}_{k,n+3}+k\,{F}_{k,n+4}\,) \\
=&({F}_{k,n}+k\,{F}_{k,n+1})+i\,({F}_{k,n+1}+k\,{F}_{k,n+2})+j\,({F}_{k,n+2}+k\,{F}_{k,n+3})+k\,({F}_{k,n+3}+k\,{F}_{k,n+4}) \\
=&{F}_{k,n+2}+i\,{F}_{k,n+3}+j\,{F}_{k,n+4}+k\,{F}_{k,n+5} \\
=&\mathbb{H}F_{k,n+2}. 
\end{array}
\end{equation*}
(\ref{E16}): By the equation (\ref{E7}) we get,
\begin{equation*}
\begin{array}{rl}
\mathbb{H}F_{k,n}^2=&(F_{k,n}^2+F_{k,n+1}^2+F_{k,n+2}^2+F_{k,n+3}^2)+2\,{F}_{k,n}(\,i\,{F}_{k,n+1}+j\,{F}_{k,n+2}+k\,{F}_{k,n+3}) \\
=&2\,{F}_{k,n}\,\mathbb{H}F_{k,n}-2\,F_{k,n}^2+(F_{k,n}^2+F_{k,n+1}^2+F_{k,n+2}^2+F_{k,n+3}^2) \\
=&2{F}_{k,n}\,\mathbb{H}F_{k,n}-{\mathbb{H}F_{k,n}}\,\,\overline{\mathbb{H}F}_{k,n}. 
\end{array}
\end{equation*}
(\ref{E17}): By the equation (\ref{E7}) we get,
\begin{equation*}
\begin{array}{rl}
\mathbb{H}F_{k,n}-i\,\mathbb{H}F_{k,n+1}-j\,\mathbb{H}F_{k,n+2}-k\,\mathbb{H}F_{k,n+3}=&{F}_{k,n}-{F}_{k,n+2}-{F}_{k,n+4}-{F}_{k,n+6}.
\end{array}
\end{equation*}
\end{proof}
\begin{theorem}
For \, $m\ge n+1$ the d'Ocagne's identity for hyperbolic k-Fibonacci quaternions ${\mathbb{H}F_{k,m}}$ and ${\mathbb{H}F_{k,n}}$ is given by
\begin{equation}\label{E18}
\mathbb{H}F_{k,m}\,\mathbb{H}F_{k,n+1}-\mathbb{H}F_{k,m+1}\,\mathbb{H}F_{k,n}=(-1)^n\,{F}_{k,m-n}\,(\,\mathbb{D}F_{k,1}+\bold{j}+k\,\bold{k}\,)=(-1)^n\,{F}_{k,m-n}\,(\,\mathbb{D}F_{k,1}+\mathbb{H}F_{k,-1}-1\,).
\end{equation}
\end{theorem}
\begin{proof}
(\ref{E18}): By using (\ref{E7})
\begin{equation*}
\begin{array}{rl}
\mathbb{H}F_{k,m}\,\mathbb{H}F_{k,n+1}-\mathbb{H}F_{k,m+1}\,\mathbb{H}F_{k,n}=&({F}_{k,m}\,{F}_{k,n+1}-{F}_{k,m+1}\,{F}_{k,n})+({F}_{k,m+1}\,{F}_{k,n+2}-{F}_{k,m+2}\,{F}_{k,n+1})\\
&+({F}_{k,m+2}\,{F}_{k,n+3}-{F}_{k,m+3}\,{F}_{k,n+2})+({F}_{k,m+3}\,{F}_{k,n+4}-{F}_{k,m+4}\,{F}_{k,n+3}) \\
&+\bold{i}\,[\,({F}_{k,m}\,{F}_{k,n+2}-{F}_{k,m+1}\,{F}_{k,n+1})+({F}_{k,m+1}\,{F}_{k,n+1}-{F}_{k,m+2}\,{F}_{k,n})\\
&+({F}_{k,m+2}\,{F}_{k,n+4}-{F}_{k,m+3}\,{F}_{k,n+3})-({F}_{k,m+3}\,{F}_{k,n+3}-{F}_{k,m+4}\,{F}_{k,n+2})\,] \\
&+\bold{j}\,[\,({F}_{k,m}\,{F}_{k,n+3}-{F}_{k,m+1}\,{F}_{k,n+2})-({F}_{k,m+1}\,{F}_{k,n+4}-{F}_{k,m+2}\,{F}_{k,n+3})\\
&+({F}_{k,m+2}\,{F}_{k,n+1}-{F}_{k,m+3}\,{F}_{k,n})+({F}_{k,m+3}\,{F}_{k,n+2}-{F}_{k,m+4}\,{F}_{k,n+1})\,] \\
&+\bold{k}\,[\,({F}_{k,m}\,{F}_{k,n+4}-{F}_{k,m+1}\,{F}_{k,n+3})+({F}_{k,m+1}\,{F}_{k,n+3}-{F}_{k,m+2}\,{F}_{k,n+2})\\
&-({F}_{k,m+2}\,{F}_{k,n+2}-{F}_{k,m+3}\,{F}_{k,n+1})+({F}_{k,m+3}\,{F}_{k,n+1}-{F}_{k,m+4}\,{F}_{k,n})\,] \\
=&(-1)^n\,\,[\,0-2\,\bold{i}\,{F}_{k,m-n-1}+2\,\bold{j}\,{F}_{k,m-n-2}+\bold{k}\,({L}_{k,m-n}+(k^3+3\,k)\,{F}_{k,m-n})\,]. 
\end{array}
\end{equation*}
Here, d'Ocagne's identity of k-Fibonacci number
${F}_{k,m}{F}_{k,n+1}-{F}_{k,m+1}{F}_{k,n}=(-1)^n\,{F}_{k,m-n}$ in Falcon and Plaza \cite{article.4} was used.
\end{proof}
\begin{theorem} 
For \, $n,m\ge0$ the Honsberger identity for hyperbolic k-Fibonacci quaternions ${\mathbb{H}F_{k,n}}$ and ${\mathbb{H}F_{k,m}}$ is given by
\begin{equation}\label{E19}
\begin{aligned}
\mathbb{H}F_{k,n+1}\,\mathbb{H}F_{k,m}+\mathbb{H}F_{k,n}\,\mathbb{H}F_{k,m-1}=&2\,\mathbb{H}F_{k,n+m}+k\,{F}_{k,n+m+1}+{L}_{k,n+m+5}. 
\end{aligned}
\end{equation}
\end{theorem}
\begin{proof}
(\ref{E19}) By using (\ref{E11})
\begin{equation*}
\begin{array}{rl}
\mathbb{H}F_{k,n+1}\,\mathbb{H}F_{k,m}=&({F}_{k,n+1}\,{F}_{k,m}+{F}_{k,n+2}\,{F}_{k,m+1}+{F}_{k,n+3}\,{F}_{k,m+2}+{F}_{k,n+4}\,{F}_{k,m+3})\\
&+\bold{i}\,({F}_{k,n+1}\,{F}_{k,m+1}+{F}_{k,n+2}\,{F}_{k,m}+{F}_{k,n+3}\,{F}_{k,m+3}-{F}_{k,n+4}\,{F}_{k,m+2})\\
&+\bold{j}\,({F}_{k,n+1}\,{F}_{k,m+2}-{F}_{k,n+2}\,{F}_{k,m+3}+{F}_{k,n+3}\,{F}_{k,m}+{F}_{k,n+4}\,{F}_{k,m+1}) \\
&+\bold{k}\,({F}_{k,n+1}\,{F}_{k,m+3}+{F}_{k,n+2}\,{F}_{k,m+2}-{F}_{k,n+3}\,{F}_{k,m+1}+{F}_{k,n+4}\,{F}_{k,m}),
\end{array}
\end{equation*}
\begin{equation*}
\begin{array}{rl}
\mathbb{H}F_{k,n}\,\mathbb{H}F_{k,m-1}=&({F}_{k,n}\,{F}_{k,m-1}+{F}_{k,n+1}\,{F}_{k,m}+{F}_{k,n+2}\,{F}_{k,m+1}+{F}_{k,n+3}\,{F}_{k,m+2})\\
&+\bold{i}\,({F}_{k,n}\,{F}_{k,m}+{F}_{k,n+1}\,{F}_{k,m-1}+{F}_{k,n+2}\,{F}_{k,m+2}-{F}_{k,n+3}\,{F}_{k,m+1})\\
&+\bold{j}\,({F}_{k,n}\,{F}_{k,m+1}-{F}_{k,n+1}\,{F}_{k,m+2}+{F}_{k,n+2}\,{F}_{k,m-1}+{F}_{k,n+3}\,{F}_{k,m}) \\
&+\bold{k}\,({F}_{k,n}\,{F}_{k,m+2}+{F}_{k,n+1}\,{F}_{k,m+1}-{F}_{k,n+2}\,{F}_{k,m}+{F}_{k,n+3}\,{F}_{k,m-1}). 
\end{array}
\end{equation*}
Finally, adding by two sides to the side,  we obtain
\begin{equation*}
\begin{array}{rl}
\mathbb{H}F_{k,n+1}\,\mathbb{H}F_{k,m}+\mathbb{H}F_{k,n}\,\mathbb{H}F_{k,m-1}=&({F}_{k,n+m}+{F}_{k,n+m+2}+{F}_{k,n+m+4}+{F}_{k,n+m+6})\\
&+2\,\bold{i}\,{F}_{k,n+m+1}+2\,\bold{j}\,{F}_{k,n+m+2}+ 2\,\bold{k}\,{F}_{k,n+m+3}\\
=&2\,(\,{F}_{k,n+m}+\bold{i}\,{F}_{k,n+m+1}+\bold{j}\,{F}_{k,n+m+2}+\bold{k}\,{F}_{k,n+m+3}\,)-{F}_{k,n+m}+{F}_{k,n+m+2}+{L}_{k,n+m+5}\,) \\ 
=&2\,\mathbb{H}F_{k,n+m}+k\,{F}_{k,n+m+1}+{L}_{k,n+m+5}.
\end{array}
\end{equation*}
Here, the Honsberger identity of k-Fibonacci number 
${F}_{k,n+1}{F}_{k,m}+{F}_{k,n}{F}_{k,m-1}={F}_{k,n+m}$ in Falcon and Plaza \cite{article.5} and ${F}_{k,n+1}+{F}_{k,n-1}={L}_{k,n}$ \cite{article.16} was used.
\end{proof}  
\begin{theorem} Let ${\mathbb{H}F_{k,n}}$ and ${\mathbb{H}L_{k,n}}$ be {n-th} terms of hyperbolic k-Fibonacci quaternion $({\mathbb{H}F_{k,n}})$ and hyperbolic k-Lucas quaternion $({\mathbb{H}L_{k,n}})$, respectively. The following relation is satisfied
\begin{equation}\label{E20}
\mathbb{H}F_{k,n+1}+\mathbb{H}F_{k,n-1}=\mathbb{H}L_{k,n}. 
\end{equation}
and
\begin{equation}\label{E21}
\mathbb{H}F_{k,n+2}-\mathbb{H}F_{k,n-2}=k\,\mathbb{H}L_{k,n}. 
\end{equation}
\end{theorem}
\begin{proof}
(\ref{E20})
From equation (\ref{E7}) and identity $F_{k,n+1}+F_{k,n-1}={L}_{k,n}$, \,$n\ge 1$ Ramirez \cite{article.16} between k-Fibonacci number and k-Lucas number, it follows that
\begin{equation*}
\begin{aligned}
\small{\begin{array}{rl}
\mathbb{H}F_{k,n+1}+\mathbb{H}F_{k,n-1}=&({F}_{k,n+1}+{F}_{k,n-1})+\bold{i}\,({F}_{k,n+2}+{F}_{k,n})+\bold{j}\,({F}_{k,n+3}+{F}_{k,n+1})+\bold{k}\,({F}_{k,n+4}+{F}_{k,n+2}) \\
=&({L}_{k,n}+\bold{i}\,{L}_{k,n+1}+\bold{j}\,{L}_{k,n+2}+\bold{k}\,{L}_{k,n+3})\\
=&\mathbb{H}L_{k,n}.
\end{array}}
\end{aligned}
\end{equation*}
(\ref{E21})
From equation (\ref{E7}) and identity $F_{k,n+2}-F_{k,n-2}=k\,{L}_{k,n}$, \,$n\ge1$ between k-Fibonacci number and k-Lucas number, it follows that
\begin{equation*}
\begin{aligned}
\small{\begin{array}{rl}
\mathbb{H}F_{k,n+2}-\mathbb{H}F_{k,n-2}=&({F}_{k,n+2}-{F}_{k,n-2})+\bold{i}\,({F}_{k,n+3}-{F}_{k,n-1})+\bold{j}\,({F}_{k,n+4}-{F}_{k,n})+\bold{k}\,({F}_{k,n+5}-{F}_{k,n+1}) \\
=&(k\,{L}_{k,n}+\bold{i}\,k\,{L}_{k,n+1}+\bold{j}\,k\,{L}_{k,n+2}+\bold{k}\,k\,{L}_{k,n+3})\\
=&k\,\mathbb{H}L_{k,n}.
\end{array}}
\end{aligned}
\end{equation*}
\end{proof}
\begin{theorem} Let $\overline{\mathbb{H}F}_{k,n}$ be conjugation of hyperbolic k-Fibonacci quaternion $({\mathbb{H}F_{k,n}})$. In this case, we can give the following relations between these quaternions:
\begin{equation}\label{E22}
\mathbb{H}F_{k,n}+\overline{\mathbb{H}F}_{k,n}=2\,{F}_{k,n}
\end{equation}
\begin{equation}\label{E23}
\mathbb{H}F_{k,n}\overline{\mathbb{H}F}_{k,n}+\mathbb{H}F_{k,n-1}\overline{\mathbb{H}F}_{k,n-1}={F}_{k,2n-1}-{F}_{k,2n+1}-{F}_{k,2n+3}-{F}_{k,2n+5}
\end{equation}
\end{theorem}
\begin{proof}
(\ref{E22}):
By using (\ref{E14}), we get
\begin{equation*}
\begin{array}{rl}
\mathbb{H}F_{k,n}+\overline{\mathbb{H}F}_{k,n}=&({F}_{k,n}+\bold{i}\,{{F}_{k,n+1}}+\bold{j}\,{F}_{k,n+2}+\bold{k}\,{F}_{k,n+3})+({F}_{k,n}-\bold{i}\,{F}_{k,n+1}-\bold{j}\,{F}_{k,n+2}-\bold{k}\,{F}_{k,n+3}\,) \\
=&2\,{F}_{k,n}. 
\end{array}
\end{equation*}
(\ref{E23}): By using (\ref{E15}), we get
\begin{equation*}
\begin{array}{rl}						
\mathbb{H}F_{k,n}\overline{\mathbb{H}F}_{k,n}+\mathbb{H}F_{k,n-1}\overline{\mathbb{H}F}_{k,n-1}=&({F}_{k,n}^2+{F}_{k,n-1}^2)-({F}_{k,n+1}^2+{F}_{k,n}^2)-({F}_{k,n+2}^2+{F}_{k,n+1}^2)-({F}_{k,n+3}^2+{F}_{k,n+2}^2)\\
=&{F}_{k,2n-1}-{F}_{k,2n+1}-{F}_{k,2n+3}-{F}_{k,2n+5}.
\end{array}
\end{equation*}
where the identity of k-Fibonacci number ${F}_{k,n}^2+{F}_{k,n+1}^2={F}_{k,2n+1}$,\,\, $n\ge0$  Ramirez \cite{article.16} was used. \\
\end{proof} 
\begin{theorem}  Let ${\mathbb{H}F_{k,n}}$ be hyperbolic k-Fibonacci quaternion. Then, we have the following identities
\begin{equation}\label{E24}
\sum\limits_{s=1}^{n}{\mathbb{H}F_{k,s}}=\frac{1}{k}\,(\mathbb{H}F_{k,n+1}+\mathbb{H}F_{k,n}-\mathbb{H}F_{k,1}-\mathbb{H}F_{k,0}),
\end{equation}
\begin{equation}\label{E25}
\sum\limits_{s=1}^{n}{\mathbb{H}F_{k,2s-1}}=\frac{1}{k}\,({\mathbb{H}F_{k,2n}}-\,{\mathbb{H}F_{k,0}}),
\end{equation}
\begin{equation}\label{E26}
\sum\limits_{s=1}^{n}{\mathbb{H}F_{k,2s}}=\frac{1}{k}\,(\mathbb{H}F_{k,2n+1}-\mathbb{H}F_{k,1}).
\end{equation}
\end{theorem}
\begin{proof}
(\ref{E24}): Since ${\sum\nolimits_{i=1}^{n}{F}_{k,i}=\frac{1}{k}({F}_{k,n+1}+{F}_{k,n}-1)}$  Falcon and Plaza \cite{article.5}, \, we get
\begin{equation*}
\begin{aligned}
\sum\limits_{s=1}^{n}{\mathbb{H}F_{k,s}}=&\sum\limits_{s=1}^{n}{{F}_{k,s}}+\bold{i}\,\sum\limits_{s=1}^{n}{F}_{k,s+1}+\bold{j}\,\sum\limits_{s=1}^{n}{F}_{k,s+2}+\bold{k}\,\sum\limits_{s=1}^{n}{F}_{k,s+3} \\
=&\frac{1}{k}\,\{\,({F}_{k,n+1}+{F}_{k,n}-1)+\bold{i}\,({F}_{k,n+2}+{F}_{k,n+1}-k-1)+\bold{j}\,[{F}_{k,n+3}+{F}_{k,n+2}-(k^2+1)-k] \\
&+\bold{k}\,[{F}_{k,n+4}+{F}_{k,n+3}-(k^3+2k)-(k^2+1)]\,\} \\
=&\frac{1}{k}\,\{({F}_{k,n+1}+\bold{i}\,{F}_{k,n+2}+\bold{j}\,{F}_{k,n+3}+\bold{k}\,{F}_{k,n+4}) \\
&+({F}_{k,n}+\bold{i}\,{F}_{k,n+1}+\bold{j}\,{F}_{k,n+2}+\bold{k}\,{F}_{k,n+3}) \\
&-[1+\bold{i}\,(k+1)+\bold{j}\,(k^2+k+1)+\bold{k}\,(k^3+k^2+2k+1)]\,\}\\
=&\frac{1}{k}\,\{\mathbb{H}F_{k,n+1}+\mathbb{H}F_{k,n}-[\,{F}_{k,1}+\bold{i}\,({F}_{k,2}+{F}_{k,1})+\bold{j}\,({F}_{k,3}+{F}_{k,2})+\bold{k}\,({F}_{k,4}+{F}_{k,3})\,]\,\} \\
=&\frac{1}{k}\,(\mathbb{H}F_{k,n+1}+\mathbb{H}F_{k,n}-\mathbb{H}F_{k,1}-\mathbb{H}F_{k,0}).
\end{aligned}
\end{equation*}
(\ref{E25}): Using ${\sum\limits_{i=1}^{n}{F}_{k,2i-1}=\frac{1}{k}\,{F}_{k,2n}}$\,, Falcon and Plaza \cite{article.5}, \, we get
\begin{equation*}
\begin{array}{rl}
 \sum\limits_{s=1}^{n}{\mathbb{H}F_{k,2s-1}}= &\frac{1}{k}\,\{\,({F}_{k,2n})+\bold{i}\,({F}_{k,2n+1}-1)+\bold{j}\,({F}_{k,2n+2}-k)+\bold{k}\,[\,{F}_{k,2n+3}-(k^2+1)\,]\} \\
=&\frac{1}{k}\,\{\,[\,{F}_{k,2n}+\bold{i}\,{F}_{k,2n+1}+\bold{j}\,{F}_{k,2n+2}+\bold{k}\,{F}_{k,2n+3}]-\,[\,\bold{i}\,+k\,\bold{j}\,+(k^2+1)\,\bold{k}\,]\,\}\\
=&\frac{1}{k}\,\{\,{\mathbb{H}F_{k,2n}}-({F}_{k,0}+\bold{i}\,{F}_{k,1}+\bold{j}\,{F}_{k,2}+\bold{k}\,{F}_{k,3})\,\} \\=&\frac{1}{k}\,({\mathbb{H}F_{k,2n}}-\,{\mathbb{H}F_{k,0}}).
\end{array}
\end{equation*}
(\ref{E26}): Using $\sum\limits_{i=1}^{n}{F}_{k,2i}=\frac{1}{k}({F}_{k,2n+1}-1)$ \, Falcon and Plaza \cite{article.5}, \, we obtain
\begin{equation*}
\begin{array}{rl}
\sum\limits_{s=1}^{n}{\mathbb{H}F_{k,2s}}= &\frac{1}{k}\{({F}_{k,2n+1}-1)+\bold{i}\,[{F}_{k,2n+2}-k]+\bold{j}\,[{F}_{2n+3}-(k^2+1)]+\bold{k}\,[{F}_{k,2n+4}-(k^3+2k)]\} \\
=&\frac{1}{k}\,\{\,({F}_{k,2n+1}+\bold{i}\,{F}_{k,2n+2}+\bold{j}\,{F}_{k,2n+3}+\bold{k}\,{F}_{k,2n+4}\,) \\
&-\,[1+k\,\bold{i}\,+(k^2+1)\,\bold{j}\,+(k^3+2k)\,\bold{k}\,]\,\} \\
=&\frac{1}{k}\,(\mathbb{H}F_{k,2n+1}-\mathbb{H}F_{k,1}). 
\end{array}
\end{equation*}
\end{proof}
\begin{theorem} \textbf{(Binet's Formula)} Let ${\mathbb{H}F_{k,n}}$ be hyperbolic k-Fibonacci quaternion. For $n\ge 1$,  Binet's formula for these quaternions is as follows Falcon and Plaza \cite{article.5}:
\begin{equation}\label{E27}
{\mathbb{H}F_{k,n}}=\frac{1}{\sqrt{k^2+4}}\left( \,\hat{\alpha }\,\,{{\alpha }^{n}}-\hat{\beta \,}\,{{\beta }^{n}} \right)\,
\end{equation}
where
\begin{equation*}
\begin{array}{l}
\hat{\alpha }=\,1+\bold{i}\,[(k-\beta)] +\bold{j}\,[\,(k^2+1)-k\,\beta\,]+\bold{k}\,[\,(k^3+2\,k)-(k^2+1)\beta\,],
\end{array}
\end{equation*}
and
\begin{equation*}
\begin{array}{l}
\hat{\beta }=-\,1+\bold{i}\,(\alpha-k) +\bold{j}\,[\,k\,\alpha-(k^2+1)\,]+\bold{k}\,[\,(k^2+1)\,\alpha-(k^3+2\,k)\,].
\end{array}
\end{equation*}
\end{theorem}
\begin{proof}
The characteristic equation of recurrence relation ${\mathbb{H}F_{k,n+2}}=k\,{\mathbb{H}F_{k,n+1}}+{\mathbb{H}F_{k,n}}$ \, \,  is
\begin{equation*}
{t}^2-k\,t-1=0.
\end{equation*}
The roots of this equation are 
\begin{equation*}
\alpha =\frac{k+\sqrt{k^2+4}}{2} \, \, \, \text{and} \, \, \, \beta =\frac{k-\sqrt{k^2+4}}{2}
\end{equation*} 
where  $\alpha +\beta =k\ ,\ \ \alpha -\beta =\sqrt{k^2+4}\,,\ \ \alpha \beta =-1$.
\par Using recurrence relation and initial values ${\mathbb{H}F_{k,0}}=(0,\,1,\,k,\,k^2+1)$, \\ ${\mathbb{H}F_{k,1}}=(1,\,k,\,k^2+1,\,k^3+2\,k)$, the Binet formula for ${\mathbb{H}F_{k,n}}$ is 	
\begin{equation*}
{\mathbb{H}F_{k,n}}=A\,{\alpha }^{n}+B\,{\beta }^{n}=\frac{1}{\sqrt{k^2+4}}\left[ \,\hat{\alpha }\,\,{\alpha }^{n}-\hat{\beta \,}\,{\beta }^{n} \right],
\end{equation*}
where $A=\cfrac{{\mathbb{H}F_{k,1}}-{\beta}\,{\mathbb{H}F_{k,0}}\,}{{\alpha }-{\beta }}$, $B=\cfrac{{\alpha }\,{\mathbb{H}F_{k,0}}-{{\mathbb{H}F_{k,1}}}}{{\alpha}-{\beta }}$ \,and  
$\hat{\alpha}=\,1+\bold{i}\,\alpha +\bold{j}\,\alpha^2+\bold{k}\,\alpha^3, \, \, \hat{\beta }=\,\,1+\bold{i}\,\beta +\bold{j}\,\beta^2+\bold{k}\,\beta^3$.
\end{proof} 

\begin{theorem} \textbf{(Cassini's Identity)}. Let ${\mathbb{H}F_{k,n}}$ be hyperbolic k-Fibonacci quaternion. For $n\ge1$,  Cassini's identity for ${\mathbb{H}F_{k,n}}$ is as follows: 
\begin{equation}\label{E28}
\mathbb{H}F_{k,n-1}\mathbb{H}F_{k,n+1}-({\mathbb{H}F_{k,n}})^\bold{2}=(-1)^{n}\,[\,2\,k\,\bold{i}+2\,(k^2+1)\,\bold{j}+(k^3+2\,k)\,\bold{k}\,].
\end{equation}
\end{theorem}
\begin{proof}
(\ref{E28}): By using (\ref{E7}) and (\ref{E11}), we get
\begin{equation*}
\begin{array}{rl}
{\mathbb{H}F_{k,n-1}}{\mathbb{H}F_{k,n+1}}-({\mathbb{H}F_{k,n}})^2=&\,[\,({F}_{k,n-1}{F}_{k,n+1}-{F}_{k,n}^2)+({F}_{k,n}{F}_{k,n+2}-{F}_{k,n+1}^2)\\
&+({F}_{k,n+1}{F}_{k,n+3}-{F}_{k,n+2}^2)+({F}_{k,n+2}{F}_{k,n+4}-{F}_{k,n+3}^2)\,]\\
&+\bold{i}\,[\,({F}_{k,n-1}{F}_{k,n+2}-{F}_{k,n}{F}_{k,n+1})-({F}_{k,n+1}{F}_{k,n+4}-{F}_{k,n+2}{F}_{k,n+3})\,] \\
&+\bold{j}\,[\,({F}_{k,n-1}{F}_{k,n+3}-{F}_{k,n}{F}_{k,n+2})-({F}_{k,n}{F}_{k,n+4}-{F}_{k,n+1}{F}_{k,n+3})\\
&+({F}_{k,n+1}{F}_{k,n+1}-{F}_{k,n+2}{F}_{k,n})+({F}_{k,n+2}{F}_{k,n+2}-{F}_{k,n+3}{F}_{k,n+1})\,] \\
&+\bold{k}\,[\,({F}_{k,n-1}{F}_{k,n+4}-{F}_{k,n}{F}_{k,n+3})+({F}_{k,n}{F}_{k,n+3}-{F}_{k,n+1}{F}_{k,n+2})\\
&+({F}_{k,n+2}{F}_{k,n+1}-{F}_{k,n+3}{F}_{k,n})\,] \\
=&(-1)^{n}\,[\,2\,k\,\bold{i}+2\,(k^2+1)\,\bold{j}+(k^3+2\,k)\,\bold{k}\,]. 
\end{array}
\end{equation*} 
Here, the identity of the k-Fibonacci number ${F}_{k,n-1}{F}_{k,n+1}-{F}_{k,n}^2=(-1)^{n}$, \, Falcon and Plaza \cite{article.5} was used.
\end{proof}

\begin{theorem} \textbf{(Catalan's Identity)}. Let  ${\mathbb{H}F_{k,n}}$ be hyperbolic k-Fibonacci quaternion. For $n\ge 1$,  Catalan's identity for ${\mathbb{H}F_{k,n}}$ is as follows: 
\begin{equation}\label{E29}
\mathbb{H}F_{k,n+r-1}\,\mathbb{H}F_{k,n+r+1}-\mathbb{H}F_{k,n+r}^2=(-1)^{n+r}\,(\mathbb{H}F_{k,1}+\mathbb{H}F_{k,-1}-1). 
\end{equation}
\end{theorem}
\begin{proof}
(\ref{E29}): By using (\ref{E7}) and (\ref{E11}), we get 
\begin{equation*}
\begin{array}{rl}
\mathbb{H}F_{k,n+r-1}\,\mathbb{H}F_{k,n+r+1}\,-\mathbb{H}F_{k,n+r}^2=&[\,({F}_{k,n+r-1}{F}_{k,n+r+1}-{F}_{k,n+r}^2)+({F}_{k,n+r}{F}_{k,n+r+2}-{F}_{k,n+r+1}^2)\\
&+({F}_{k,n+r+1}{F}_{k,n+r+3}-{F}_{k,n+r+2}^2)+({F}_{k,n+r+2}{F}_{k,n+r+4}-{F}_{k,n+r+3}^2)\,] \\
&+\bold{i}\,[\,({F}_{k,n+r-1}{F}_{k,n+r+2}-\,{F}_{k,n+r}{F}_{k,n+r+1})+({F}_{k,n+r+1}{F}_{k,n+r+4}-\,{F}_{k,n+r+2}{F}_{k,n+r+3})\,] \\
&+\bold{j}\,[\,({F}_{k,n+r-1}{F}_{k,n+r+3}-{F}_{k,n+r}{F}_{k,n+r+2})+({F}_{k,n+r+1}{F}_{k,n+r+1}-{F}_{k,n+r+2}{F}_{k,n+r})\\
&-({F}_{k,n+r}{F}_{k,n+r+4}-{F}_{k,n+r+1}{F}_{k,n+r+3})+({F}_{k,n+r+2}{F}_{k,n+r+2}-{F}_{k,n+r+3}{F}_{k,n+r+1})\,] \\ 
&+\bold{k}\,[\,({F}_{k,n+r-1}{F}_{k,n+r+4}-{F}_{k,n+r}{F}_{k,n+r+3})+({F}_{k,n+r+2}{F}_{k,n+r+1}-{F}_{k,n+r+3}{F}_{k,n+r})\\
&+({F}_{k,n+r}{F}_{k,n+r+3}-{F}_{k,n+r+1}{F}_{k,n+r+2})\,] \\
=&(-1)^{n+r}\,[0+2\,k\,\bold{i}\,+2\,(k^2+1)\,\bold{j}\,+(k^3+2\,k)\bold{k}\,].
\end{array}
\end{equation*}
Here, the identity of the k-Fibonacci number ${F}_{k,n+r-1}{F}_{k,n+r+1}-{F}_{k,n+r}^2=(-1)^{n+r}$,\, Falcon and Plaza \cite{article.5} was used.
\end{proof}
\section{Conclusion}
In this study, a number of new results on hyperbolic k-Fibonacci quaternions are derived. 
\newpage

%%% ENTER REFERENCES IN THE FORM

\end{document}